\documentclass[12pt,oneside,reqno]{amsart}

\usepackage{amsmath,amssymb,amsthm,textcomp}
\usepackage{amsfonts,graphicx}
\usepackage[mathscr]{eucal}
\usepackage{color}
\usepackage{csquotes}
\usepackage[backend=bibtex,%
firstinits=true,%
doi=false,%
isbn=true,%
url=false,%
maxnames=99]{biblatex}%

\AtEveryBibitem{\clearfield{issn}}
\AtEveryCitekey{\clearfield{issn}}
\numberwithin{equation}{section}
\DeclareNameAlias{sortname}{last-first}
\theoremstyle{definition}

\numberwithin{equation}{section}

\newcommand{\ncom}{\newcommand}

\ncom{\beq}{\begin{equation}}
\ncom{\eeq}{\end{equation}}
\ncom{\bea}{\begin{eqnarray*}}
\ncom{\eea}{\end{eqnarray*}}
\ncom{\beqa}{\begin{eqnarray}}
\ncom{\eeqa}{\end{eqnarray}}
\ncom{\nno}{\nonumber}
\ncom{\non}{\nonumber}
\ncom{\ds}{\displaystyle}
\ncom{\half}{\frac{1}{2}}
\ncom{\mbx}{\makebox{.25cm}}
\ncom{\hs}{\mbox{\hspace{.25cm}}}
\ncom{\rar}{\rightarrow}
\ncom{\Rar}{\Rightarrow}
\ncom{\noin}{\noindent}
\ncom{\bc}{\begin{center}}
\ncom{\ec}{\end{center}}
\ncom{\sz}{\scriptsize}
\ncom{\rf}{\ref}
\ncom{\s}{\sqrt{2}}
\ncom{\sgm}{\sigma}
\ncom{\Sgm}{\Sigma}
\ncom{\psgm}{\sigma^{\prime}}
\ncom{\dt}{\delta}
\ncom{\Dt}{\Delta}
\ncom{\lmd}{\lambda}
\ncom{\Lmd}{\Lambda}
\ncom{\Th}{\Theta}
\ncom{\e}{\eta}
\ncom{\eps}{\epsilon}
\ncom{\pcc}{\stackrel{P}{>}}
\ncom{\lp}{\stackrel{L_{p}}{>}}
\ncom{\dist}{{\rm\,dist}}
\ncom{\sspan}{{\rm\,span}}
\ncom{\re}{{\rm Re\,}}
\ncom{\im}{{\rm Im\,}}
\ncom{\sgn}{{\rm sgn\,}}
\ncom{\ba}{\begin{array}}
\ncom{\ea}{\end{array}}
\ncom{\hone}{\mbox{\hspace{1em}}}
\ncom{\htwo}{\mbox{\hspace{2em}}}
\ncom{\hthree}{\mbox{\hspace{3em}}}
\ncom{\hfour}{\mbox{\hspace{4em}}}
\ncom{\vone}{\vskip 2ex}
\ncom{\vtwo}{\vskip 4ex}
\ncom{\vonee}{\vskip 1.5ex}
\ncom{\vthree}{\vskip 6ex}
\ncom{\vfour}{\vspace*{8ex}}
\ncom{\norm}{\|\;\;\|}
\ncom{\integ}[4]{\int_{#1}^{#2}\,{#3}\,d{#4}}
\ncom{\vspan}[1]{{{\rm\,span}\{ #1 \}}}
\ncom{\dm}[1]{ {\displaystyle{#1} } }
\ncom{\ri}[1]{{#1} \index{#1}}

\newtheorem{theorem}{\bf Theorem}[section]

\newtheoremstyle
    {remarkstyle}
    {}
    {11pt}
    {}
    {}
    {\bfseries}
    {:}
    {     }
    {\thmname{#1} \thmnumber{#2} }

\theoremstyle{remarkstyle}

\def\eps{\varepsilon}

\begin{document}
\title{A Probabilistic Proof of The Multinomial Theorem}
\author[Kuldeep Kumar Kataria]{Kuldeep Kumar Kataria}
\address{Kuldeep Kumar Kataria, Department of Mathematics,
 Indian Institute of Technology Bombay, Powai, Mumbai 400076, INDIA.}
 \email{kulkat@math.iitb.ac.in}
\thanks{The research of K. K. Kataria was supported by UGC, Govt. of India.}
\keywords{Mathematical induction; Combinatorial analysis; Multinomial distribution.}
\begin{abstract}
In this note, we give an alternate proof of the multinomial theorem using a probabilistic approach. Although the multinomial theorem is basically a combinatorial result, our proof may be simpler for a student familiar with only basic probability concepts.
\end{abstract}

\maketitle
\section{Introduction}
The multinomial theorem is an important result with many  applications in mathematical statistics and computing. It expands $(x_1+x_2+\ldots+x_m)^n$, for integer $n\geq0$, into the sum of the products of integer powers of real numbers $x_1,x_2,\ldots,x_m$. The prevalent proofs of the multinomial theorem are either based on the principle of mathematical induction (see [2, pp. 78-80]) or on counting arguments (see [1, p. 33]). The proof by induction make use of the binomial theorem and is a bit complicated. Rosalsky [4] provided a probabilistic proof of the binomial theorem using the binomial distribution. Indeed, we use the multinomial distribution (see [3, 197-198]) to prove the multinomial theorem. For students familiar with basic probability theory, our proof may be simpler than the existing proofs by mathematical induction and principles of combinatorics. Also, this proof avoids the use of the binomial theorem. The multinomial distribution is modeled as follows.\\
Consider an experiment consisting of $n$ independent trials. The outcome of each trial results in the occurrence of one of the $m$ mutually exclusive and exhaustive events $E_1,E_2,\ldots,E_m$. For each $i=1,2,\ldots,m$, let $p_i$ be the constant probability of the occurrence of the event $E_i$ and $X_i$ be the random variable that denotes the number of times event $E_i$ has occurred. Then, the joint probability mass function of the random variables $X_1,X_2,\ldots,X_m$ is
\begin{equation}\label{1.1}
P\left\{X_1=k_1,X_2=k_2,\ldots,X_m=k_m\right\}=n!\prod_{j=1}^m\frac{p_j^{k_j}}{k_j!},
\end{equation}
where $\sum_{i=1}^mk_i=n$. Also, since (\ref{1.1}) is a valid statistical distribution, we have
\begin{equation}\label{1.2}
1=\sum_{\sum_{i=1}^mk_i=n}n!\prod_{j=1}^m\frac{p_j^{k_j}}{k_j!}.
\end{equation}
Next, we state and prove the multinomial theorem.
\section{A probabilistic proof of the multinomial theorem}
\begin{theorem}[Multinomial]
Let $n$ be any non negative integer and $x_1,x_2$, $\ldots,x_m$ be real numbers. Then 
\begin{equation*}\label{t2.1}
(x_1+x_2+\ldots+x_m)^n=\sum_{\sum_{i=1}^mk_i=n}\frac{n!}{k_1!k_2!\ldots k_m!}x_1^{k_1}x_2^{k_2}\ldots x_m^{k_m},
\end{equation*}
where $k_i$'s are non negative integers and $0^0$ is interpreted as unity.
\end{theorem}

\begin{proof}
Let us consider
\begin{equation*}\label{tt2.1}
(x_1+x_2+\ldots+x_m)^n=\underset{n\ \mathrm{times}}{\underbrace{(x_1+x_2+\ldots+x_m)\ldots(x_1+x_2+\ldots+x_m)}}.
\end{equation*}
By using the distributive property in the right hand side of the above equation, it follows for all real numbers $x_i$'s that
\begin{equation}\label{2.1}
(x_1+x_2+\ldots+x_m)^n=\sum_{\sum_{i=1}^mk_i=n}C(n,k_1,k_2,\ldots, k_m)x_1^{k_1}x_2^{k_2}\ldots x_m^{k_m},
\end{equation}
where $C(n,k_1,k_2,\ldots k_m)$ are positive integers and $k_i$'s are non negative integers satisfying $\sum_{i=1}^mk_i=n$. We just need to show that
\begin{equation}\label{2.2}
C(n,k_1,k_2,\ldots, k_m)=\frac{n!}{k_1!k_2!\ldots k_m!}.
\end{equation}
Assume $x_i>0$ for all $i=1,2,\ldots,m$ and define
\begin{equation}\label{2.3}
p_i=\frac{x_i}{x_1+x_2+\ldots+x_m}.
\end{equation}
Clearly, $0<p_i<1$ and $\sum_{i=1}^{m}p_i=1$.
On substituting (\ref{2.3}) in (\ref{1.2}), we obtain for positive reals
\begin{equation}\label{2.4}
(x_1+x_2+\ldots+x_m)^n=\sum_{\sum_{i=1}^mk_i=n}\frac{n!}{k_1!k_2!\ldots k_m!}x_1^{k_1}x_2^{k_2}\ldots x_m^{k_m}.
\end{equation}
Finally, subtracting (\ref{2.4}) from (\ref{2.1}),
\begin{equation}\label{2.5}
\sum_{\sum_{i=1}^mk_i=n}\left(C(n,k_1,k_2,\ldots, k_m)-\frac{n!}{k_1!k_2!\ldots k_m!}\right)x_1^{k_1}x_2^{k_2}\ldots x_m^{k_m}=0,\ \ x_i>0.
\end{equation}
Since (\ref{2.5}) is a zero polynomial in $m$ variables, (\ref{2.2}) follows and the proof is complete.
\end{proof}

\end{document}